\documentclass{amsproc}
\usepackage{amsfonts}

\theoremstyle{plain}
\newtheorem{theorem}{Theorem}[section]

\newtheorem{definition}[theorem]{Definition}
\newtheorem{example}[theorem]{Example}

\newtheorem{lemma}[theorem]{Lemma}
\newtheorem{notation}[theorem]{Notation}

\newtheorem{proposition}[theorem]{Proposition}
\newtheorem{remark}[theorem]{Remark}

\numberwithin{equation}{section}
\numberwithin{table}{section}

\begin{document}
\title[Assouad dimensions of measures]{Measures with specified support and
arbitrary Assouad dimensions }
\author{Kathryn E. Hare}
\address{Dept. of Pure Mathematics, University of Waterloo, Waterloo, Ont.,
Canada, N2L 3G1}
\email{kehare@uwaterloo.ca}
\author{Franklin Mendivil}
\address{Department of Mathematics and Statistics, Acadia University,
Wolfville, N.S. Canada, B4P 2R6}
\email{franklin.mendivil@acadiau.ca}
\author{Leandro Zuberman}
\address{Centro Marplatense de Investigaciones Matem\'{a}ticas, \\
Universidad Nacional de Mar del Plata, Argentina and CONICET}
\email{leandro.zuberman@gmail.com}
\thanks{The research of K. Hare is partially supported by NSERC 2016:03719.
The research of F. Mendivil is partially supported by NSERC\ 2012:238549. K.
Hare and L. Zuberman thank Acadia University for their hospitality when some
of this research was done.}
\thanks{This paper is in final form and no version of it will be submitted
for publication elsewhere.}

\begin{abstract}
We show that if the upper Assouad dimension of the compact set $E\subseteq 
\mathbb{R}$ is positive, then given any $D>\dim _{A}E$ there is a measure
with support $E$ and upper Assouad (or regularity) dimension $D$. Similarly,
given any $0\leq d<\dim _{L}E,$ there is a measure on $E$ with lower Assouad
dimension $d$.
\end{abstract}

\maketitle

\section{Introduction}

The upper and lower Assouad dimensions of measures are analogues of the
Assouad dimensions of metric spaces and provide quantitative information
about the extreme local behaviour of the measures. These dimensions were
extensively studied by Fraser and Howroyd in \cite{FH} and K\"{a}enm\"{a}ki
et al.\ in \cite{KL} and \cite{KLV}, where they were called the upper and
lower regularity dimensions. It is known that the upper Assouad dimension of
a measure $\mu ,$ $\dim _{A}\mu ,$ is finite if and only if the measure is
doubling and that the upper Assouad dimension of $\mu $ always dominates the
upper Assouad dimension of the support of $\mu $. Similarly, a measure $\mu $
has positive lower Assouad dimension, $\dim _{L}\mu $, if and only if $\mu $
is uniformly perfect and the lower Assouad dimension of a doubling measure
is always dominated by the lower Assouad dimension of its support.

In \cite{VK}, Volberg and Konyagin proved that any compact, doubling, metric
space $E$ supports a measure whose upper Assouad dimension is arbitrarily
close to that of the space $E$. This was extended to complete doubling
spaces in \cite{LS}. The analogue for lower Assouad dimension was proved by
Byland and Gudayol in \cite{BG} and K\"{a}enm\"{a}ki and Lehrb\"{a}ck in 
\cite{KL}. Volberg and Konyagin also showed that there is a complete,
doubling, metric space $E$ with the property that every measure supported on 
$E$ has upper Assouad dimension strictly greater than that of $E$.

Fraser and Howroyd in \cite{FH} proved that if $E$ is the closure of the
discrete set $\{q^{n}:n=1,2,...\}\subseteq \mathbb{R}$ where $q<1,$ then the
measure $\mu =\sum p^{n}\delta _{q^{n}}$, where $p<1,$ has upper Assouad
dimension $\log p/\log q$. Thus there are measures $\mu $ with support $E$
and $\dim _{A}\mu $ equal to any given value greater than $\dim _{A}E$
(which is zero, in this case). It is natural to ask if this property holds
more generally. The purpose of this note is to show that it does hold for
any compact subset of $\mathbb{R}$ with positive upper Assouad dimension.
More specifically, if $\dim _{A}E>0,$ then given any $D>\dim _{A}E$
(including $D=\infty )$ there is a measure $\mu $ with support $E$ and $\dim
_{A}\mu =D$. A key ingredient in our argument is the generalised nested
cubes construction of K\"{a}enm\"{a}ki et al in \cite{KRS}.

Similarly, given any $0\leq d<\dim _{L}E$ there is a measure $\nu $ with
support $E$ and $\dim _{L}\nu =d$. If $d>0,$ we can achieve both specified
dimensions with a single measure. We also show that these results need not
be true if we only assume $E$ is a compact, infinite subset of $\mathbb{R}$.

\section{Definitions}

Given $E$, a compact subset of a metric space, we write $N_{r}(E)$ for the
least number of sets of diameter at most $r$ that are required to cover $E$.

\begin{definition}
The upper and lower Assouad dimensions of $E,$ denoted $\dim _{A}E$ and $%
\dim _{L}E$ respectively, are given by 
\begin{equation*}
\dim _{A}E=\inf \left\{ \alpha :(\exists C_{1},C_{2}>0)(\forall
0<r<R<C_{1})\ \sup_{x\in E}N_{r}(B(x,R)\cap E)\leq C_{2}\left( \frac{R}{r}%
\right) ^{\alpha }\right\}
\end{equation*}%
and%
\begin{equation*}
\dim _{L}E=\sup \left\{ \alpha :(\exists C_{1},C_{2}>0)(\forall
0<r<R<C_{1})\ \inf_{x\in E}N_{r}(B(x,R)\cap E)\geq C_{2}\left( \frac{R}{r}%
\right) ^{\alpha }\right\}
\end{equation*}
\end{definition}

One can similarly define the Assouad dimensions of any finite, positive
Borel measure defined on the metric space.

\begin{definition}
The upper and lower Assouad dimensions of measure $\mu ,$ denoted $\dim
_{A}\mu $ and $\dim _{L}\mu $ respectively, are given by 
\begin{equation*}
\dim _{A}\mu =\inf \left\{ \alpha :(\exists C_{1},C_{2}>0)(\forall 0<r\leq
R\leq C_{1})\sup_{x\in supp\mu }\frac{\mu (B(x,R))}{\mu (B(x,r))}\leq
C_{2}\left( \frac{R}{r}\right) ^{\alpha }\right\}
\end{equation*}%
and 
\begin{equation*}
\dim _{L}\mu =\sup \left\{ \alpha :(\exists C_{1},C_{2}>0)(\forall
0<r<R<C_{1})\inf_{x\in supp\mu }\frac{\mu (B(x,R))}{\mu (B(x,r))}\geq
C_{2}\left( \frac{R}{r}\right) ^{\alpha }\right\}
\end{equation*}
\end{definition}

The upper and lower Assouad dimensions of a measure have also been referred
to as the upper and lower regularity dimensions.

The Assouad dimensions quantify the extreme local behaviour of the set or
measure. Indeed, the following relationships are known for all compact
metric spaces $E$ and measures $\mu $:%
\begin{equation*}
\dim _{L}E\leq \dim _{H}E\leq \underline{\dim }_{B}E\leq \overline{\dim }%
_{B}E\leq \dim _{A}E
\end{equation*}%
and 
\begin{equation*}
\dim _{L}\mu \leq \dim _{H}\mu \leq \dim _{A}\mu \text{,}
\end{equation*}%
where $\dim _{H}E,$ \underline{$\dim $}$_{B}E$ and $\overline{\dim }_{B}E$
denote the Hausdorff, lower and upper box dimensions of $E$ respectively and 
$\dim _{H}\mu $ is the Hausdorff dimension of $\mu $. The upper Assouad
dimension of a subset of $\mathbb{R}^{n}$ is at most $n$. More generally,
the upper Assouad dimension of $E$ is finite if and only if $E$ is a
doubling metric space. Similarly, the upper Assouad dimension of measure $%
\mu $ is finite if and only if the measure is doubling, meaning there is a
constant $C$ such that 
\begin{equation*}
\mu (B(x,R))\leq C\mu (B(x,R/2))
\end{equation*}%
whenever $x$ belongs to the support of $\mu $ and $R>0$. Unlike the case for
sets, the upper Assouad dimension of a measure on $\mathbb{R}^{n}$ can be
infinite. Dually, the lower Assouad dimension of a set\ (or measure) is
positive if and only if the set (resp., measure) is uniformly perfect, which
for measures means there are constants $C>1,\tau >1$ such that%
\begin{equation*}
\mu (B(x,\tau R))\geq C\mu (B(x,R))
\end{equation*}%
for all $x\in $supp$\mu $ and $0<R\leq diam($supp$\mu )/(2\tau )$. It is not
difficult to see that for any measure $\mu $, 
\begin{equation*}
\dim _{A}\text{supp}\mu \leq \dim _{A}\mu ,
\end{equation*}%
and if $\mu $ is a doubling measure, then%
\begin{equation*}
\dim _{L}\mu \leq \dim _{L}\text{supp}\mu .
\end{equation*}%
For more background on these dimensions and proofs of these basic facts we
refer the reader to \cite{FTrans}-\cite{KLV}.

As the upper Assouad dimension of a union of two sets is the maximum of
their upper Assouad dimensions, one might expect that the upper Assouad
dimension of the sum of two measures is their maximum upper Assouad
dimension. This is not true as the following example illustrates.

\begin{example}
Measures $\mu ,\nu $ with $\dim _{A}(\mu +\nu )<\max (\dim _{A}\mu ,\dim
_{A}\nu )$: The two measures will be supported on the classical middle third
Cantor set. Pick $0<p<1/2$. We will define the measures by specifying them
on the Cantor intervals of the standard Cantor set construction. We label
these intervals as $I_{w}$ where $w$ is a finite word on the letters $0,1$,
where $I_{w0}$ is the left child of interval $I_{w}$ and $I_{w1}$ is the
right child.

Let $w$ be such a word. If $w=0^{(n)}$, then define $\mu (I_{w})=p^{n}$. If $%
w=0^{(n)}1\sigma $ where $n\geq 0$ and $\sigma $ is a word of length $%
\left\vert \sigma \right\vert =k$ on the letters $0,1,$ then put $\mu
(I_{w})=p^{n}(1-p)2^{-k}$. The measure $\nu $ will be defined similarly, but
with the roles of $0,1$ interchanged. It is easy to see that $\mu ,\nu $ are
probability measures.

As $p<1/2,$ if $w=0\sigma ,$ then $\mu (I_{w})\leq (1-p)2^{-|\sigma |}$,
while $\nu (I_{w})=(1-p)2^{-|\sigma |}$. Consequently, $(\mu +\nu
)(I_{w})=O(2^{-|w|})$. Similarly, if $w=1\sigma ,$ then $\nu (I_{w})\leq
(1-p)2^{-|\sigma |}$, while $\mu (I_{w})=(1-p)2^{-|\sigma |}$, so again, $%
(\mu +\nu )(I_{w})=O(2^{-|w|})$. It follows easily from this that $\dim
_{A}(\mu +\nu )=\log 2/\log 3$.

But, 
\begin{equation*}
\frac{\mu (I_{0^{(n)}})}{\mu (I_{0^{(n+k)}})}=p^{-k}\text{ while }\frac{%
|I_{0^{(n)}}|}{|I_{0^{(n+k)}}|}=3^{k},
\end{equation*}%
thus $\dim _{A}\mu \geq |\log p|/\log 3>\log 2/\log 3$. (In fact, we have
equality.) A similar statement holds for $\nu $.
\end{example}

\section{Upper Assouad Dimension}

We begin by reviewing the\ construction introduced by K\"{a}enm\"{a}ki,
Rajala and Suomala in \cite{KRS}, which we will call the KRS\ construction.
This construction will be used in the remainder of the paper to construct
measures which have the required upper and/or lower Assouad dimensions. We
remark that we have made slight changes to some of their notation and we
will also introduce further terminology that will be helpful for our
application.

\textbf{KRS Construction:} By the KRS\ construction of a set $E\subseteq
\lbrack 0,1]$, we will mean the nested, dyadic-like `cubes' construction
described in \cite{KRS}. This construction was shown to hold for any metric
space which is doubling. In the case of subsets of $\mathbb{R}$, the
generalized cubes of their construction can be taken to be intervals.

To be more specific, K\"{a}enm\"{a}ki et al have shown that for any $0<s<1/3$
and for each $k\in \mathbb{N}$ there are disjoint intervals, $I_{w},$
labelled by words $w$ of length $\left\vert w\right\vert =$ $k,$ with 
\begin{equation*}
\bigcup_{w:\left\vert w\right\vert =k}I_{w}\bigcap E=E
\end{equation*}%
and which have the properties listed below:

Properties. (i) There are constants $c,C>0$ such that if $\left\vert
w\right\vert =$ $k$ (we call $k$ the \textbf{level }of $I_{w}$), then 
\begin{equation*}
2cs^{k}\leq \ell (I_{w})\leq 2Cs^{k},
\end{equation*}%
where $\ell (I_{w})$ denotes the length of the interval. By choosing $s$
sufficiently small, we can assume $c=3/8$ and $C=9/8$, so $C=3c$. We will
write 
\begin{equation*}
\text{KRS}(s,c,C)
\end{equation*}%
if we want to emphasize the particular choice of parameters $s,c,C$.

(ii) Each $I_{w}$ contains a \textbf{distinguished point,} $x_{w},$ such
that $d(x_{w},I_{w}^{c})\geq cs^{k}$ if $\left\vert w\right\vert =k.$

(iii) Each $I_{w}$ has at least one child. Moreover, $I_{w}$ decomposes into
its children, $I_{wj},$ $j=1,...,N_{w},$ where $N_{w}$ $\leq N=N(s,E)$. By
this, we mean that $I_{wj}\subseteq I_{w}$, $I_{wi}\bigcap I_{wj}=\phi $ if $%
i\neq j$ and 
\begin{equation*}
I_{w}\cap E=\bigcup_{j=1}^{N_{w}}(I_{wj}\cap E).
\end{equation*}

We will call $I_{wj}$ a \textbf{boundary child} (or \textbf{boundary interval%
}) of the interval $I_{w}$ of level $k$ if $d(I_{wj},I_{w}^{c})<cs^{k+1}$.
As $\ell (I_{wj})\geq 2cs^{k+1},$ there can be at most two boundary children
for each $I_{w},$ one on each side of $I_{w}$. Any child of $I_{w}$ that is
not a boundary child will be called an \textbf{interior child} (or \textbf{%
interior interval}). Two boundary intervals that do not have the same parent
(that is, they are not \textbf{siblings}) will be called \textbf{cousins}.
We will write $I_{w^{-}}$ for the \textbf{parent} of $I_{w}$.

We will say that two level $k$ intervals, $I_{w}$ and $I_{v},$ are \textbf{%
adjacent} if the distance between them is at most $cs^{k}$. This ensures
there is no interval of level $k$ between them and consequently the open
interval in-between contains no points of $E$. If adjacent intervals are not
siblings, they must both be boundary intervals of (different) parents that
are adjacent to each other. Any interval can be adjacent to at most two
other intervals (one on each side). At most one can fail to be a sibling.

Notice that the distinguished point, $x_{w},$ does not belong to a boundary
child of $I_{w}$. Consequently every $I_{w}$ has a child that is an interior
interval. The interior child of $I_{w}$ containing $x_{w}$ will be called
the \textbf{distinguished} \textbf{interior child }/ \textbf{interval}.

We will say $I_{w}$ \textbf{splits} if it has at least two children. Notice
that if $I_{w}$ has a boundary child it must split.

By a \textbf{path} of length $L$, we will mean a sequence of nested
intervals, 
\begin{equation*}
I_{w},I_{wj_{1}},I_{wj_{1}j_{2},}...,I_{wu},
\end{equation*}%
where $\left\vert u\right\vert =L$. Here each interval in the sequence is a
child of the previous interval. We will call it a\textbf{\ boundary path} if
all the descendents of $I_{w}$ are boundary intervals. A boundary path is a 
\textbf{splitting path} (meaning each interval, other than possibly the
last, splits) because $I_{w}$ and each $I_{wj_{1}...j_{k}}$ (except possibly 
$I_{wu})$ has a boundary child and hence at least two children.

\begin{lemma}
\label{Lem:s} Let $E\subseteq \lbrack 0,1]$ and $\varepsilon >0$. There is a
choice of $s(\varepsilon )<1/3$ and integer $k(\varepsilon )$ such that the
following holds: Take any KRS$(s,c,3c)$ construction of $E$ with $s\leq
s(\varepsilon )$. If $N_{w}$ is the number of children of interval $I_{w}$
and $\left\vert w\right\vert \geq k(\varepsilon ),$ then 
\begin{equation*}
s^{-(\dim _{L}E-\varepsilon )}\leq N_{w}\leq s^{-(\dim _{A}E+\varepsilon )}.
\end{equation*}
\end{lemma}

\begin{proof}
From the definition of the upper Assouad dimension we can find constants $A$
and $\rho >0,$ depending on $\varepsilon ,$ such that for all $r\leq R\leq
\rho $ and $x\in E,$ we have 
\begin{equation}
N_{r}(B(x,R))\leq A\left( \frac{R}{r}\right) ^{\dim _{A}E+\varepsilon /2}%
\text{ .}  \label{A1}
\end{equation}%
Choose $s(\varepsilon )$ so small that $A(3s^{-1})^{\dim _{A}E+\varepsilon
/2}\leq s^{-(dim_{A}E+\varepsilon )}$ for all $s\leq s(\varepsilon )$.
Choose $k(\varepsilon )$ such that $2Cs(\varepsilon )^{k}<\rho $ for all $%
k\geq k(\varepsilon )$ where $C=3c$.

Take any interval $I_{w}$ from a KRS$(s,c,3c)$\ construction of $E$ where $%
\left\vert w\right\vert \geq k\geq k(\varepsilon )$. Let $x_{w}$ be the
distinguished point of $I_{w}.$ Consider $B(x_{w},2Cs^{k})\supseteq I_{w}$.
As the children of $I_{w}$ have length at least $2cs^{k+1},$ if $r=2cs^{k+1}$
we clearly have%
\begin{equation*}
N_{r}(B(x_{w},2Cs^{k}))\geq N_{w}.
\end{equation*}%
But from (\ref{A1}), we also have 
\begin{eqnarray*}
N_{r}(B(x_{w},2Cs^{k})) &\leq &A\left( \frac{2Cs^{k}}{2cs^{k+1}}\right)
^{\dim _{A}E+\varepsilon /2} \\
&=&A(3s^{-1})^{\dim _{A}E+\varepsilon /2}\leq s^{-(dim_{A}E+\varepsilon )}.
\end{eqnarray*}%
Consequently,%
\begin{equation*}
N_{w}\leq s^{-(\dim _{A}E+\varepsilon )}.
\end{equation*}

Similarly, choose positive constants $A^{\prime },\rho ^{\prime },$
depending on $\varepsilon ,$ such that 
\begin{equation*}
N_{r}(B(x,R))\geq A^{\prime }\left( \frac{R}{r}\right) ^{\dim
_{L}E-\varepsilon /2}
\end{equation*}
for all $r\leq R\leq \rho ^{\prime }$. As $B(x_{w},cs^{k})\subseteq I_{w},$
we have 
\begin{equation*}
N_{w}\geq N_{2Cs^{k+1}}(B(x_{w},cs^{k}))\geq A^{\prime }\left( \frac{cs^{k}}{%
2Cs^{k+1}}\right) ^{\dim _{L}E-\varepsilon /2}\geq s^{-(\dim
_{L}E-\varepsilon )}\text{ }
\end{equation*}%
for large enough $k,$ provided we choose $s\leq s(\varepsilon )$ and $%
s(\varepsilon )>0$ sufficiently small.
\end{proof}

The main goal of this section is to prove the following theorem.

\begin{theorem}
\label{upper}Let $E\subseteq \lbrack 0,1]$ be compact. If $\dim _{A}E>0,$
then given any $D>\dim _{A}E,$ there is a probability measure $\mu $ with
support $E$ and $\dim _{A}\mu =D$.
\end{theorem}

Our basic strategy is to use the KRS\ construction to define a probability
measure $\mu $, with support equal to $E,$ by specifying $\mu $ on each $%
I_{w}$ in a consistent fashion. Indeed, we will assign positive weights, $%
\Delta (I_{wj}),$ so that $\sum_{j=1}^{N_{w}}\Delta (I_{wj})=1$ and define 
\begin{equation*}
\mu (I_{wj})=\Delta (I_{wj})\mu (I_{w}).
\end{equation*}%
The task in proving the theorem is to define these weights appropriately.

There are two cases to consider for the proof of the theorem. The easier
case is when there are arbitrarily long boundary paths, in which case there
are also arbitrarily long boundary paths beginning at arbitrarily deep
levels.

\begin{lemma}
\label{longbdy}Assume $\dim _{A}E<D<\infty $ and let $\varepsilon =(D-\dim
_{A}E)/2$. If there is a KRS$(s,c,3c)$ construction of $E,$ with the choice
of $s\leq s(\varepsilon )$ from Lemma \ref{Lem:s}, which admits arbitrarily
long boundary paths, then there is a probability measure $\mu $ with support 
$E$ and $\dim _{A}\mu =D$.
\end{lemma}

\begin{proof}
For each $j>k(\varepsilon )$, as defined in Lemma \ref{Lem:s}, choose a
boundary path beginning at level $n_{j}$ and having length $l_{j}\geq j,$
where $n_{j+1}>n_{j}+l_{j}$. Assume this is the path 
\begin{equation*}
P_{j}:=I_{v}^{(j)},I_{vj_{1}}^{(j)},I_{vj_{1}j_{2}}^{(j)},...,I_{vv^{\prime
}}^{(j)},
\end{equation*}%
where $\left\vert v\right\vert =n_{j}$ and $\left\vert v^{\prime
}\right\vert =l_{j}$. Consider, also, the adjacent boundary path (should it
exist)%
\begin{equation*}
I_{w}^{(j)},I_{wi_{1}}^{(j)},I_{wi_{1}i_{2}}^{(j)},...,I_{ww^{\prime }}^{(j)}
\end{equation*}%
where each $I_{vj_{1}j_{2}...j_{k}}^{(j)}$ is adjacent to $%
I_{wi_{1}i_{2}...i_{k}}^{(j)}$ and $\left\vert w^{\prime }\right\vert \leq
\left\vert v^{\prime }\right\vert $ (the inequality being strict if for some 
$m<\left\vert v^{\prime }\right\vert ,$ the interval $%
I_{wi_{1}...i_{m}}^{(j)}$ has no boundary child adjacent to $%
I_{vj_{1}...j_{m+1}}^{(j)}$). Notice that while $I_{v}^{(j)}$ and $%
I_{w}^{(j)}$ may be siblings, all other pairings are cousins.

We define the measure $\mu $ as follows: The weights associated with each of
the intervals $I_{vj_{1}}^{(j)},I_{vj_{1}j_{2},}^{(j)}$, $...,I_{vv^{\prime
}}^{(j)}$ and $I_{wi_{1}}^{(j)},I_{wi_{1}i_{2},}^{(j)}$, $...,I_{ww^{\prime
}}^{(j)}$ will be the value $a=s^{D}$. Each sibling of these intervals,
other than the distinguished sibling, will be assigned weight $%
p=s^{D-\varepsilon }>a$. Each of the distinguished siblings (which are not
any of the intervals in the paths) will be assigned the weight which ensures
that summing over all children of a given parent is $1$. This weight will be
at least $p$ since, as shown in Lemma \ref{Lem:s}, $p=s^{D-\varepsilon }\leq
1/(\max N_{w})$. The paths arising from different $j$ come from different
levels, hence this specification is well defined.

This procedure assigns weights to (precisely) all the children of each of $%
I_{v}^{(j)},I_{vj_{1}}^{(j)}$, $I_{vj_{1}j_{2},}^{(j)}$, $...,I_{(vv^{\prime
})^{-}}^{(j)}$ and $I_{w}^{(j)},I_{wi_{1}}^{(j)}$, $I_{wi_{1}i_{2},}^{(j)}$, 
$...,I_{(ww^{\prime })^{-}}^{(j)}$. For all remaining intervals, $I_{\sigma
},$ of level at least $k(\varepsilon )$, we assign weight $p$ to all
children of the parent $I_{\sigma }$, except the distinguished child which
is assigned the weight $1-p(N_{\sigma }-1)$, so the sum of the weights of
all the children of $I_{\sigma }$ is $1$. We also assign equal weights
(adding to $1$) to the intervals of level $k(\varepsilon )$. This defines
the probability measure $\mu $.

One important observation is that $a$ is the minimum weight assigned to any
interval.

Another important observation is that any two adjacent intervals that are
not siblings are either both assigned weight $a$ or both assigned weight $p$%
, and in either case, have equal weights. To see why this is relevant,
suppose that $I_{\sigma }$ and $I_{\eta }$ are adjacent cousins. Let $%
I_{\tau }$ be their nearest common ancestor and assume that $I_{\tau
},I_{\tau i_{1}},...,I_{\tau i_{1}...i_{l}}=I_{\sigma }$ and $I_{\tau
},I_{\tau j_{1}},...,I_{\tau j_{1}...j_{l}}=I_{\eta }$ are paths. Then $%
I_{\tau i_{1}}\neq I_{\tau j_{1}}$. Furthermore, while $I_{\tau i_{1}}$ and $%
I_{\tau j_{1}}$ may be siblings, all other pairs $I_{\tau i_{1}...i_{m}},$ $%
I_{\tau j_{1}...j_{m}}$ are adjacent cousins as they do not have the same
parent. Consequently, their weights are the same. Hence%
\begin{equation*}
\frac{\mu (I_{\sigma })}{\mu (I_{\eta })}=\frac{\mu (I_{\tau })\Delta
(I_{\tau i_{1}})\cdot \cdot \cdot \Delta (I_{\sigma })\text{ }}{\mu (I_{\tau
})\Delta (I_{\tau j_{1}})\cdot \cdot \cdot \Delta (I_{\eta })}=\frac{\Delta
(I_{\tau i_{1}})\text{ }}{\Delta (I_{\tau j_{1}})}.
\end{equation*}%
As $\Delta (I_{\tau i_{1}}),\Delta (I_{\tau j_{1}})\in \lbrack a,1],$ 
\begin{equation}
a\mu (I_{\eta })\leq \mu (I_{\sigma })\leq \frac{1}{a}\mu (I_{\eta })
\label{cousins}
\end{equation}%
for all adjacent cousins. A similar statement holds for the measures of any
two siblings.

We are now ready to compute the upper Assouad dimension of $\mu $. First,
consider any $r\leq R$ and $x\in E$. Choose integers $J,K$ such that 
\begin{equation*}
\frac{c}{2}s^{J+1}\leq R<\frac{c}{2}s^{J}\text{ }
\end{equation*}%
and 
\begin{equation*}
2Cs^{K}\leq r<2Cs^{K-1}\text{.}
\end{equation*}%
Since $2Cs^{K}\leq (c/2)s^{J}$ and $C>c,$ we have $K>J$. Choose $v,w=v\tau $
such that $\left\vert v\right\vert =J,$ $\left\vert w\right\vert =K$ and $%
x\in I_{w}\subseteq I_{v}$. As $\ell (I_{w})\leq 2Cs^{K},$ we have $%
I_{w}\subseteq B(x,r)$, and so $\mu (B(x,r))\geq \mu (I_{w})$. Since the
minimum value of any weight is $a$, it follows that $\mu (I_{w})\geq
a^{K-J}\mu (I_{v})$.

Consider $B(x,R)$. If this set was to intersect three intervals of level $J$%
, then it would contain one of them. But $2R<cs^{J}<$ length of any interval
of level $J$, hence this is impossible. Thus it can intersect at most two
such intervals, one of which must be $I_{v}$.

If $B(x,R)$ only intersects one interval of level $J$, then $B(x,R)\cap
E\subseteq I_{v}$. Hence $\mu (B(x,R))\leq \mu (I_{v})$ and 
\begin{equation*}
\frac{\mu (B(x,R))}{\mu (B(x,r))}\leq \frac{\mu (I_{v})}{a^{K-J}\mu (I_{v})}%
=a^{J-K}=s^{D(J-K)}\leq A_{1}\left( \frac{R}{r}\right) ^{D}
\end{equation*}%
for a suitable constant $A_{1}$.

If $B(x,R)$ intersects two intervals, $I_{v}$ and $I_{\sigma }$, then the
distance between $I_{v}$ and $I_{\sigma }$ can be at most the diameter of
the ball, $2R<cs^{J},$ so these are adjacent intervals. As noted in (\ref%
{cousins}), this means $\mu (I_{v})+\mu (I_{\sigma })\leq (1+1/a)\mu (I_{v})$%
. Thus we also have

\begin{equation*}
\frac{\mu (B(x,R))}{\mu (B(x,r))}\leq \frac{\mu (I_{v})+\mu (I_{\sigma })}{%
a^{K-J}\mu (I_{v})}\leq \frac{(1+\frac{1}{a})\mu (I_{v})}{a^{K-J}\mu (I_{v})}%
\leq A_{2}\left( \frac{R}{r}\right) ^{D}.
\end{equation*}%
It follows from these computations that $\dim _{A}\mu $ $\leq D$.

To see the equality, consider one of the special boundary paths, 
\begin{equation*}
P_{j}:=I_{v}^{(j)},I_{vj_{1}}^{(j)},...,I_{vv^{\prime }}^{(j)},
\end{equation*}%
where $I_{v}^{(j)}$ is of level $n_{j}$ and $I_{vv^{\prime
}}^{(j)}=I_{w}^{(j)}$ is of level $n_{j}+l_{j}$. Then $\mu (I_{w}^{(j)})=\mu
(I_{v}^{(j)})a^{l_{j}}$. Let $x_{j}$ be the distinguished point of $%
I_{w}^{(j)}$ and choose $r_{j}=cs^{n_{j}+l_{j}}$ so that $%
B(x_{j},r_{j})\subseteq I_{w}^{(j)}$. Take $R_{j}=2Cs^{n_{j}}$ so that $%
B(x_{j},R_{j})\supseteq I_{v}^{(j)}$. Then 
\begin{equation*}
\frac{\mu (B(x_{j},R_{j}))}{\mu (B(x_{j},r_{j}))}\geq \frac{\mu (I_{v}^{(j)})%
}{\mu (I_{w}^{(j)})}\geq \frac{\mu (I_{v}^{(j)})}{a^{l_{j}}\mu (I_{v}^{(j)})}%
=s^{-Dl_{j}}\geq A_{3}\left( \frac{R_{j}}{r_{j}}\right) ^{D}.
\end{equation*}%
As $R_{j}\rightarrow 0$ and $R_{j}/r_{j}\rightarrow \infty $ as $%
j\rightarrow \infty ,$ it follows that $\dim _{A}\mu \geq D,$ proving the
equality.
\end{proof}

The more difficult case is when there are not arbitrarily long boundary
paths. For this, it is convenient to introduce further terminology.

\begin{notation}
Assume we have fixed a KRS construction. For each $n\in \mathbb{N}$, let%
\begin{eqnarray*}
\zeta _{n} &=&\sup \{\varepsilon :\exists \text{ path beginning at level }%
\geq \text{ }n\text{ and length }L\geq n\text{ s.t. } \\
&&\text{at least }\varepsilon L\text{ of the intervals in the path split}\}.
\end{eqnarray*}%
Of course, $\zeta _{n}\in \lbrack 0,1]$ and decreases as $n$ increases. We
will set 
\begin{equation*}
\zeta =\lim_{n}\zeta _{n}
\end{equation*}%
and call $\zeta $ the \textbf{proportionality constant }of the KRS
construction.
\end{notation}

\begin{lemma}
\label{proportional}For any set $E$ of positive upper Assouad dimension and
any KRS\ construction with parameter $s<1/3,$ the proportionality constant $%
\zeta $ is strictly positive.
\end{lemma}

\begin{remark}
Under the much stronger hypothesis, $\dim _{L}E>0,$ we actually have $\zeta
=1$ for small enough $s$. This can be seen in the proof of Theorem \ref%
{lower}.
\end{remark}

\begin{proof}
Let $0<t<\dim _{A}E$. We can find $x_{n}\in E,$ $R_{n}$ and $r_{n}<R_{n}$
with $R_{n}\rightarrow 0$ and $R_{n}/r_{n}\rightarrow \infty ,$ such that 
\begin{equation*}
N_{r_{n}}(B(x_{n},R_{n})\cap E)\geq n\left( \frac{R_{n}}{r_{n}}\right) ^{t}.
\end{equation*}%
Choose integer $k_{n}$ such that 
\begin{equation}
s^{-(k_{n}-1)}<R_{n}/r_{n}\leq s^{-k_{n}}.  \label{kn}
\end{equation}%
Without loss of generality we can assume that $R_{n}<s^{n}$ and $k_{n}\geq n$%
.

Temporarily fix $n$ and choose the largest integer $N$ such that $%
cs^{N}/2>R_{n}$. At most two level $N$ intervals from the KRS$(s,c,C)$
construction can intersect $B(x_{n},R_{n})$ and hence one of them, call it $%
I_{w},$ must satisfy%
\begin{equation*}
N_{r_{n}}(I_{w}\cap E)\geq \frac{n}{2}\left( \frac{R_{n}}{r_{n}}\right) ^{t}.
\end{equation*}%
Consider the optimal covering of $I_{w}$ by intervals of radius $r_{n}$. By
ordering these intervals left-to-right in $\mathbb{R}$ and keeping every
third interval, we see there are at least%
\begin{equation*}
\frac{n}{6}\left( \frac{R_{n}}{r_{n}}\right) ^{t}\geq \frac{n}{6}%
s^{-(k_{n}-1)t}
\end{equation*}%
disjoint intervals $\{J_{l}\}$ of radius $r_{n}$ and separated by at least $%
r_{n}$, each of which contains at least one point of $E\cap I_{w}$.

Let $M$ be the smallest integer such that $2Cs^{M}<r_{n}$ and for each $%
J_{l} $ choose a subinterval $I_{wu_{l}}$ of level $M,$ which intersects $%
J_{l}$ and no other $J_{m}$. The definitions of $N,M$ ensure that%
\begin{equation*}
\frac{cs^{N+1}}{4Cs^{M-1}}\leq \frac{R_{n}}{r_{n}}\leq \frac{cs^{N}}{4Cs^{M}}%
.
\end{equation*}%
Combined with (\ref{kn}), this implies%
\begin{equation*}
\frac{\log (4C/c)}{\log s}+M-N-2\leq k_{n}\leq \frac{\log (4C/c)}{\log s}%
+M-N+1,
\end{equation*}%
hence for large enough $n,$ $n\leq k_{n}\leq M-N+1\leq 2k_{n}$.

Now consider all the paths beginning at level $N$ with interval $I_{w}$ and
ending at level $M$, hence of length $M-N$. As each KRS\ interval can have
at most $(C/c)s^{-1}$ children, an easy induction argument shows that if
every such path had at most $\varepsilon (M-N)$ splittings, then there would
be a total of at most $\left( (C/c)s^{-1}\right) ^{\varepsilon (M-N)}$ of
these paths. But as the intervals $I_{wu_{l}}$ are all distinct, there are
at least $ns^{-(k_{n}-1)t}/6$ of these paths. Since $M-N\sim k_{n},$ we
deduce that $\varepsilon \geq \gamma t$ for a suitable constant $\gamma >0.$
It follows that $\zeta >0$.
\end{proof}

We are now ready to complete the proof of Theorem \ref{upper}.

\begin{proof}
\lbrack of Theorem \ref{upper}] First, suppose $D<\infty $, say $D=\dim
_{A}E $ $+2\varepsilon $ for $\varepsilon >0$. Choose $s=s(\varepsilon )$ as
given in Lemma \ref{Lem:s} and consider a KRS$(s,c,3c)$ construction of $E$.
If this construction admits arbitrarily long boundary paths, then Lemma \ref%
{longbdy} can be applied to produce a measure supported on $E$ with upper
Assouad dimension $D$.

Thus we can assume there exists some $M>0$ such that there are no boundary
paths of length exceeding $M$. We will take $a>0$ such that $a^{\zeta
}=s^{D},$ where $\zeta $ is the proportionality constant of this KRS
construction, guaranteed to be positive by Lemma \ref{proportional}. Notice
that Lemma \ref{Lem:s} gives 
\begin{equation*}
a\leq s^{D}<s^{D-\varepsilon }\leq 1/(\max N_{w}).
\end{equation*}

We begin by defining the measure. Choose $n_{1}>M$ such that $\left\vert
\zeta _{n_{1}}-\zeta \right\vert <\zeta /2$. Choose a path $P_{1}$ beginning
at level at least $n_{1}$ and with length $L_{1}\geq $ $n_{1},$ which splits
on at least $\,L_{1}\zeta /2$ levels. The choice of $\zeta _{n_{1}}$ also
ensures no path beginning at level at least $n_{1}$ and length $L_{1}\geq $ $%
n_{1}$ splits on more than $L_{1}3\zeta /2$ levels.

Assuming $n_{1}<n_{2}<\cdot \cdot \cdot <n_{j-1}$ and paths $%
P_{1},..,P_{j-1} $ have been identified (with the properties as described in
the induction step outlined below), choose $n_{j}$ much greater than the
level of the last interval in path $P_{j-1},$ such that $\left\vert \zeta
_{n_{j}}-\zeta \right\vert <\zeta 2^{-j}$. Choose a path $P_{j}$ beginning
at level at least $n_{j}$, length $L_{j}\geq $ $n_{j},$ and which splits on
at least $L_{j}\zeta (1-2^{-j})$ levels. Furthermore, notice no such path
can split on more than the proportion $\zeta (1+2^{-j})$ levels. That is the
induction step.

Next, consider these special paths $%
P_{j}:I_{v},I_{vi_{1}},...,I_{vi_{1},..,i_{L_{j}}}$. If the interval $%
I_{vi_{1}...i_{l}}$ is an only child$,$ we assign weight $1$ to it.
Otherwise, we assign it the weight $a$ and distribute the remaining $1-a$
equally among its siblings. Notice these remaining weights will have value
at least $a$.

For all remaining intervals in the KRS\ construction (those that are not
siblings of intervals in one of the special paths), we simply assign equal
weight to all children of a given parent so that the sum of the weights of
the sibling group is $1$. This completes the definition of the measure $\mu $%
. It is clear from the construction that the support of $\mu $ is the set $E 
$. It is important to observe that$\ a$ is the minimum weight.

To show that $\dim _{A}\mu =D,$ we will actually verify that for every $%
\varepsilon >0$ there are constants $C^{\prime },C^{\prime \prime }$ and $%
\rho $ (depending on $\varepsilon )$ such that for all $r\leq R\leq \rho $
and $x\in E$, we have 
\begin{equation}
\frac{\mu (B(x,R))}{\mu (B(x,r))}\leq C^{\prime }\left( \frac{R}{r}\right)
^{D(1+\varepsilon )}  \label{1}
\end{equation}%
and that for some sequence $r_{n}\leq R_{n}\leq \rho $ and $x_{n}\in E$ with 
$R_{n}/r_{n}\rightarrow \infty ,$ we have 
\begin{equation}
\frac{\mu (B(x_{n},R_{n}))}{\mu (B(x_{n},r_{n}))}\geq C^{\prime \prime
}\left( \frac{R_{n}}{r_{n}}\right) ^{D(1-\varepsilon )}.  \label{2}
\end{equation}

Fix $\varepsilon >0$ and choose $J$ such that $2^{-J}<\varepsilon $. Put $%
\rho =(c/2)s^{n_{J}}$. Take any $x\in E$, $r\leq R\leq \rho $ and assume 
\begin{equation*}
\frac{c}{2}s^{j+1}\leq R<\frac{c}{2}s^{j}\text{ and }2Cs^{k}\leq r<2Cs^{k-1}
\end{equation*}%
for suitable $j\geq n_{J}$ and $k>j.$

As in the proof of Lemma \ref{longbdy}, there are intervals $I_{w}\subseteq
I_{v}$ such that $x\in I_{w},$ $\left\vert w\right\vert =k$, $\left\vert
v\right\vert =j$, $B(x,r)\supseteq I_{w}$ and $B(x,R)\subseteq I_{v}\cup
I_{\sigma }$ where $I_{\sigma }$ is also of level $j$ and is adjacent to $%
I_{v}$. If $I_{v}$ and $I_{\sigma }$ are siblings, their measures are
comparable. So assume their common ancestor, $I_{\tau },$ is $m$ generations
earlier. Since the path from this ancestor to either of $I_{v}$ or $%
I_{\sigma }$ is a boundary path (except possibly at the first descendent), $%
m\leq M+1$. But then%
\begin{equation*}
\mu (I_{v})\geq a^{m}\mu (I_{\tau })\geq a^{M+1}\mu (I_{\tau })\geq
a^{M+1}\mu (I_{\sigma }),
\end{equation*}%
so again their measures are comparable. Hence%
\begin{equation*}
\mu (B(x,R))\leq (1+a^{-(M+1)})\mu (I_{v}).
\end{equation*}

Now consider the path from $I_{v}$ to $I_{w}$. This path begins at level at
least $n_{J}$. If the path length, $k-j,$ is at least $n_{J},$ then by the
definition of $\zeta _{n_{J}}$ there would be at most $\zeta (1+2^{-J})(j-k)$
levels where the path splits. On the levels where it does not split, the
assigned weight is $1$ and, of course, all weights are at least $a$. Hence,
in this case, 
\begin{equation*}
\mu (B(x,r))\geq \mu (I_{w})\geq a^{\zeta (1+2^{-J})(j-k)}\mu (I_{v}).
\end{equation*}%
Consequently, for suitable constants $C_{1},C_{2}$ we have 
\begin{eqnarray*}
\frac{\mu (B(x,R))}{\mu (B(x,r))} &\leq &\frac{(1+a^{-(M+1)})\mu (I_{v})}{%
a^{\zeta (1+2^{-J})(k-j)}\mu (I_{v})}\leq C_{1}a^{\zeta (1+2^{-J})(j-k)} \\
&\leq &C_{1}s^{D(1+2^{-J})(j-k)}\leq C_{2}\left( \frac{R}{r}\right)
^{D(1+\varepsilon )}.
\end{eqnarray*}

Otherwise, $k-j<n_{J}$ and then we have $\mu (B(x,r))\geq a^{k-j}\mu
(I_{v})\geq a^{n_{J}}\mu (I_{v})$. Thus%
\begin{equation*}
\frac{\mu (B(x,R))}{\mu (B(x,r))}\leq \frac{(1+a^{-(M+1)})\mu (I_{v})}{%
a^{n_{J}}\mu (I_{v})}\leq C_{3}\leq C_{3}\left( \frac{R}{r}\right) ^{D}.
\end{equation*}%
for some constant $C_{3}$. This establishes (\ref{1}).

To see (\ref{2}), consider the special path $P_{j}$ for $j\geq J$. Suppose
this path begins at $I_{v}^{(j)}$ of level $M_{j}$ $\geq $ $n_{j}$ and ends
at $I_{w}^{(j)},$ having length $L_{j}\geq n_{j}.$ This path splits on at
least $\zeta (1-2^{-j})L_{j}\geq \zeta (1-\varepsilon )L_{j}$ levels (and
the weights are $a$ on the splitting children, $1$ on the non-splitting
levels). Thus 
\begin{equation*}
\frac{\mu (I_{v}^{(j)})}{\mu (I_{w}^{(j)})}\geq a^{-\zeta (1-\varepsilon
)L_{j}}.
\end{equation*}%
Take $x_{j}$ to be the distinguished point of $I_{w}^{(j)}$ and let $%
r_{j}=cs^{M_{j}+L_{j}},$ so $B(x_{j},r_{j})\subseteq I_{w}^{(j)}$. Take $%
R_{j}=\left\vert I_{v}^{(j)}\right\vert \sim s^{M_{j}}$. Then $%
B(x_{j},R_{j})\supseteq I_{v}^{(j)},$ thus%
\begin{equation*}
\frac{\mu (B(x_{j},R_{j}))}{\mu (B(x_{j},r_{j}))}\geq \frac{\mu (I_{v}^{(j)})%
}{\mu (I_{w}^{(j)})}\geq a^{-\zeta (1-\varepsilon
)L_{j}}=s^{-D(1-\varepsilon )L_{j}}\geq C_{4}\left( \frac{R_{j}}{r_{j}}%
\right) ^{D(1-\varepsilon )}.
\end{equation*}%
Since $R_{j}/r_{j}\sim s^{-L_{j}}\rightarrow \infty $, this establishes (\ref%
{2}) and that completes the proof for $D<\infty .$

\medskip

Now suppose $D=\infty $. Choose a sequence $\{D_{n}\}\rightarrow \infty $
with $\dim _{A}E<D_{n}<\infty $. Assume $D_{1}=\dim _{A}E+2\varepsilon $ and
take a KRS$(s(\varepsilon ),c,3c)$\ construction.

If this KRS construction admits arbitrarily long boundary paths, choose
paths $P_{j}$ as in the proof of Lemma \ref{longbdy} and assign weights $%
a_{j}$ $=s^{D_{j}}$ to the intervals in the paths $P_{j}$. All other
intervals will be assigned equal weights chosen to ensure that the sum over
all children of a given parent is $1$. Let $\mu $ be the associated measure.
This measure may no longer have the property that the measure of adjacent
intervals are comparable, but this is not necessary for obtaining the lower
bound on the upper Assouad dimension. As seen in the final paragraph of the
proof of Lemma \ref{longbdy}, it is enough to consider the starting and
ending intervals on the path $P_{j},$ say $I_{v}^{(j)}$ and $I_{w}^{(j)},$
and the measure of the balls $B(x_{j},R_{j})\supseteq I_{v}^{(j)}$ and $%
B(x_{j},r_{j})\subseteq I_{w}^{(j)},$ to deduce that $\dim _{A}\mu \geq
D_{j} $. Since $D_{j}\rightarrow \infty ,$ we conclude that $\dim _{A}\mu
=\infty $.

Otherwise, there are not arbitrarily long boundary paths. Then we modify the
argument given in the first part of the proof of this theorem in a similar
fashion. Instead of assigning weight $a=s^{D/\zeta }$, we assign weight $%
a_{j}=s^{D_{j}/\zeta }$ to the intervals in the paths $P_{j}$ that split and
assign the remaining intervals the suitable weights so that the sum of all
children of a given parent is $1$. Again we do not need the comparability of
the measure of adjacent intervals as consideration of the intervals in the
special paths is enough to see that $\dim _{A}\mu =\infty $.
\end{proof}

\begin{remark}
We remark that the same proof works for any compact set $E\subseteq \mathbb{R%
}$ for which there is a KRS\ construction with positive proportionality
constant and this can happen without $\dim _{A}E>0$. For example, let $E$ be
the closure of the set $\{\alpha ^{n}:n=0,1,2,...\}$ with $\alpha <1$. We
can take a KRS\ construction with $s=\alpha $ that has $\zeta _{n}=1$ for
all $n$. It is shown in \cite{FH} that if $\mu $ is the discrete measure $%
\sum_{n=1}^{\infty }\alpha ^{\lambda n}\delta _{\alpha ^{n}}$ having support 
$E,$ then $\dim _{A}\mu =\lambda $. Our methods also give this.
\end{remark}

If, however, $\zeta =0$ for all KRS constructions, then it need not be the
case that the set $E$ supports measures with all possible upper Assouad
dimensions. One trivial example is to take $E$ to be a finite set. Then any
measure supported on $E$ will have upper Assouad dimension zero.

More interestingly, one can construct infinite compact sets which support
only measures with upper Assouad dimension either zero or infinity, with
both values arising.

\begin{proposition}
There is an infinite compact subset $E$ of $\mathbb{R}$ such that%
\begin{equation*}
\{\dim _{A}\mu :\text{support }\mu =E\}=\{0,\infty \}.
\end{equation*}
\end{proposition}

\begin{proof}
Let $\alpha <1$ and $M>1$. Define $x_{0}=0$ and $x_{n}=\alpha ^{M^{n}}$ for $%
n\in \mathbb{N}$. Set $E=\{x_{n}:n\geq 0\}$. We will see this set has the
desired properties. Note that any measure supported on $E$ is of the form $%
\mu =\sum_{n=0}^{\infty }p(n)\delta _{x_{n}}$ for a nonegative, summable
sequence $\{p(n)\}_{n=0}^{\infty }$.

Choose $n$ large enough that $3x_{n}<x_{n-1}$ and $x_{n}>2x_{n+1}$. Then $%
\mu (B(x_{n},2x_{n}))=\mu ([0,3x_{n}))\geq p(0),$ while $\mu
(B(x_{n},x_{n}/2))=\mu ((x_{n}/2,3x_{n}/2))=p(n)$. Since $p(n)\rightarrow 0$
as $n\rightarrow \infty ,$ it follows that if $p(0)\neq 0,$ then $\mu $ is
not doubling and hence its Assouad dimension is infinite.

Thus we now assume $p(0)=0$. Put $t_{N}=\sum_{n=N}^{\infty }p(n)$. We will
distinguish between three cases.

Case (i). There is a subsequence with $\lim_{j\rightarrow \infty
}t_{n_{j}}/t_{n_{j}+1}=1$. Pick $n_{j}$ so that $3x_{n_{j}}<x_{n_{j}-1}$ and 
$x_{n_{j}}>2x_{n_{j}+1}$. Then $\mu (B(x_{n_{j}},2x_{n_{j}}))=t_{n_{j}}$ and 
\begin{equation*}
\mu (B(x_{n_{j}},x_{n_{j}}/2))=p(n_{j})=t_{n_{j}}-t_{n_{j}+1}.
\end{equation*}
Since the subsequence satisfies $t_{n_{j}}/(t_{n_{j}}-t_{n_{j}+1})%
\rightarrow \infty $, we again conclude that $\dim _{A}\mu =\infty .$

Case (ii). There is a subsequence with $\lim_{j\rightarrow \infty
}t_{n_{j}}/t_{n_{j}+1}=\infty $. If $2x_{n_{j}}<x_{n_{j}-1}$ we have that $%
\mu (B(0,2x_{n_{j}}))=t_{n_{j}}$ and $\mu (B(0,x_{n_{j}}))=t_{n_{j}+1}$.
Thus 
\begin{equation*}
\frac{\mu (B(0,2x_{n_{j}}))}{\mu (B(0,x_{n_{j}}))}=\frac{t_{n_{j}}}{%
t_{n_{j}+1}}
\end{equation*}
and again we deduce that $\dim _{A}\mu =\infty .$

Case (iii). Otherwise, there are constants $1<\lambda \leq \Lambda $ such
that $\lambda \leq t_{n}/t_{n+1}\leq \Lambda $ for all $n$. Let $\rho
=x_{1}-x_{2}$. Fix $x\in E$ and $0<r<R\leq \rho $. If $x=x_{1},$ then choose
the minimal $L\geq 1$ such that $x_{L+1}\leq x-r$. Then 
\begin{equation}
\frac{\mu (B(x,R))}{\mu (B(x,r))}\leq \frac{t_{1}}{t_{1}-t_{L+1}}\leq \frac{%
t_{1}}{t_{1}(1-t_{L+1}/t_{1})}\leq \frac{1}{1-\lambda ^{-1}}.  \label{x1}
\end{equation}%
Otherwise, $x\leq x_{2}$ and we can choose $N\geq 1$ such that $%
x_{N+1}<x+R\leq x_{N}$, so $\mu (B(x,R))\leq \mu ([0,x_{N}])\leq t_{N}$.
Choose $K$ such that $x\leq x_{K}<x+r\leq x_{K-1}$. If $x-r>0$, pick $L\geq
K $ such that $x_{L+1}\leq x-r<x_{L}$ $\leq x$. In this case, 
\begin{equation*}
\mu (B(x,r))=\mu ((x-r,x+r))\geq \mu ([x_{L},x_{K}])=t_{K}-t_{L+1}\geq
t_{K}(1-\lambda ^{-1}),
\end{equation*}%
so we have the bound 
\begin{equation}
\frac{\mu (B(x,R))}{\mu (B(x,r))}\leq \frac{t_{N}}{(1-\lambda ^{-1})t_{K}}%
\leq (1-\lambda ^{-1})^{-1}\Lambda ^{K-N}.  \label{B/b}
\end{equation}%
If $x-r\leq 0$, then $\mu (B(x,r))\geq t_{K}$ and a similar estimate is
still valid.

Put $\gamma :=\log \Lambda /\log M>0$. From the relations $\alpha
^{M^{N+1}}=x_{N+1}\leq x+R$ and $x+r\leq x_{K-1}=\alpha ^{M^{K-1}}$ we
obtain 
\begin{equation*}
\Lambda ^{-N}=M^{-\gamma N}\leq \left( \frac{M\log \alpha }{\log (x+R)}%
\right) ^{\gamma }\text{ and }\Lambda ^{K}=M^{\gamma K}\leq \left( \frac{%
M\log (x+r)}{\log \alpha }\right) ^{\gamma }.
\end{equation*}%
Using (\ref{B/b}) and these bounds on $\Lambda ^{-N}$ and $\Lambda ^{K}$ we
deduce that 
\begin{equation*}
\frac{\mu (B(x,R))}{\mu (B(x,r))}\leq C_{0}\left( \frac{\log (x+r)}{\log
(x+R)}\right) ^{\gamma }
\end{equation*}%
for a suitable constant $C_{0}$.

Consider the function $F(x)=\frac{\log (x+r)}{\log (x+R)}$, whose derivative
is 
\begin{equation*}
\frac{(x+r)^{-1}\log (x+R)-(x+R)^{-1}\log (x+r)}{\log ^{2}(x+R)}.
\end{equation*}%
As the function $\varphi (t)=t\log t$ is decreasing if $t<e^{-1},$ $%
F^{\prime }<0$ for small enough $R$. Thus $F(x)^{\gamma }\leq F(0)^{\gamma }$
for small $x\geq 0$. Moreover, for any $s>0,$ the function $\varphi
_{s}(t)=t^{s}\log t$ is decreasing if $t<e^{-1/s}$ and hence $\log r/\log
R\leq \left( R/r\right) ^{s}$ for small enough $R$. Consequently%
\begin{equation*}
\frac{\mu (B(x,R))}{\mu (B(x,r))}\leq C_{0}\left( \frac{\log r}{\log R}%
\right) ^{\gamma }\leq C_{0}\left( \frac{R}{r}\right) ^{\gamma s},
\end{equation*}%
for any $s>0.$ Together with (\ref{x1}) this proves $\dim _{A}\mu =0$.
\end{proof}

\section{Lower Assouad Dimension}

The corresponding result for lower Assouad dimension is much easier. We can
even arrange to simultaneously achieve both the prespecified upper and lower
Assouad dimensions when the lower Assouad dimension of $E$ is positive.

\begin{theorem}
\label{lower}Given $0<d<\dim _{L}E\leq \dim _{A}E<D<\infty ,$ there is a
measure $\mu $ with support $E$, $\dim _{L}\mu =d$ and $\dim _{A}\mu =D.$
\end{theorem}

\begin{proof}
Fix $0<\varepsilon <\min (\dim _{L}E-d,D-\dim _{A}E)/2$. Choose $s$%
\thinspace $>0$ so small that $s^{-\varepsilon }-2s^{d}\geq 1$, $%
s^{\varepsilon }+s^{d}<1$, $s^{-(d+\varepsilon )}\geq 3$ and $s\leq
s(\varepsilon )$ from Lemma \ref{Lem:s}. Consider a KRS($s,c,3c)$
construction of $E$ with this choice of $s$. According to Lemma \ref{Lem:s},
for all long enough words $w,$ 
\begin{equation}
s^{-(D-\varepsilon )}\geq N_{w}\geq s^{-(d+\varepsilon )}\geq 3.  \label{Nw}
\end{equation}%
Thus every (deep enough) interval from the KRS\ construction splits and if
we put $p=s^{D-\varepsilon }$, then $s^{D}\leq p\leq 1/N_{w}$ for all such $%
w $.

Now choose a rapidly increasing sequence $\{n_{j}\},$ with $n_{1}$
sufficiently large that (\ref{Nw}) holds for $\left\vert w\right\vert \geq
n_{1}$. Select an infinite path consisting of interior intervals.

For even values of $j,$ we will assign weight $s^{d}$ to the interior
intervals of the path at levels $n_{j},...,n_{j+1}-1$. The siblings of each
of these intervals will be assigned equal weights so the sum of the weights
over all the children of a given parent is $1$. We will denote these weights
by $z_{w}$ if the interior interval is $I_{w}$.

For $j$ odd, we do the same thing, but with weight $s^{D}$ assigned to the
interior intervals of the path and denote the weights assigned to siblings
of $I_{w}$ as $y_{w}$.

All boundary intervals from the KRS construction that are not siblings of
the intervals in the path will be assigned weight $p.$ We will assign equal
weights to the siblings of these boundary intervals or to sibling groups
which have no sibling boundary interval so that the sum of the weights of
any sibling group is $1$. Call these weights $x_{w}$ (where $I_{w}$ is any
member of the sibling group). Note that $x_{w}\geq p\geq s^{D}$ as $p\leq
1/N_{w}$ and $x_{w}\leq 1/(N_{w}-\eta _{w})$ where $\eta _{w}=$\# boundary
siblings of $I_{w}$. Since $N_{w}\geq s^{-(d+\varepsilon )}$ and $%
s^{-\varepsilon }-2s^{d}\geq 1,$ it is easy to see that $x_{w}\leq s^{d}$.

One can similarly use the bounds on $N_{w}$ and the relations $%
s^{-\varepsilon }-2s^{d}\geq 1$ and $s^{\varepsilon }+s^{d}<1$ to see that $%
s^{D}\leq y_{w},z_{w}\leq s^{d}$. For example, 
\begin{equation*}
z_{w}=\frac{1-s^{d}}{N_{w}-1}\geq \frac{1-s^{d}}{s^{-(D-\varepsilon )}-1}%
\geq s^{D}
\end{equation*}%
since $1+s^{D}\geq s^{\varepsilon }+s^{d}$. Thus $s^{D}$ is the minimim
weight assigned and $s^{d}$ the maximum.

We also observe that a boundary interval will never have a child that is
assigned weight $s^{d}$ or $s^{D},$ hence all boundary intervals that are
children of boundary intervals are assigned the weight $p$.

Suppose $I_{w}$ and $I_{v}$ are adjacent intervals at a given level. If they
are siblings, their $\mu $-measures are obviously comparable. If they are
not siblings, the paths from their nearest common ancestor to $I_{w}$ and $%
I_{v}$ consists of only boundary intervals after the first descendent.
Consequently, the intervals in these paths all have weight $p$ except
possibly the first two generations, hence the $\mu $-measures of adjacent
intervals are comparable.

Since there are arbitrarily long paths with weights $s^{d}$ and $s^{D}$
respectively, standard arguments now show that $\dim _{L}\mu =d$ and $\dim
_{A}\mu =D$.
\end{proof}

\begin{remark}
We note that it is easy to create a measure with given support $E$ and lower
Assouad dimension $0$. Just take any measure $\mu $ with support $E$ and
suppose $e\in E$. Then $dim_{L}(\mu +\delta _{e})=0$. Observe that if $e$ is
not an isolated point in $E,$ then $\dim _{A}(\mu +\delta _{e})=\infty $.
\end{remark}



\end{document}